\documentclass{amsart}
\usepackage{amssymb}
\usepackage{amsfonts}

\newtheorem{theorem}{Theorem}
\theoremstyle{plain}

\newtheorem{definition}{Definition}

\newtheorem{lemma}{Lemma}

\newtheorem{remark}{Remark}

\numberwithin{equation}{section}
\input{tcilatex}

\begin{document}
\title[rough multilinear fractional integral{\ operators} ]{On the behaviors
of rough multilinear fractional integral{\ }and multi-sublinear fractional
maximal{\ operators both on product }$L^{p}$ and {\ weighted }$L^{p}$ spaces}
\author{FER\.{I}T G\"{U}RB\"{U}Z}
\address{HAKKARI UNIVERSITY, FACULTY OF EDUCATION, DEPARTMENT OF MATHEMATICS
EDUCATION, HAKKARI, TURKEY }
\email{feritgurbuz84@hotmail.com}
\email{feritgurbuz@hakkari.edu.tr}
\urladdr{}
\thanks{}
\curraddr{ }
\urladdr{}
\thanks{}
\date{}
\dedicatory{}
\thanks{}

\begin{abstract}
The aim of this paper is to get the product $L^{p}$-estimates, weighted
estimates and two-weighted estimates for rough multilinear fractional
integral{\ operators }and rough multi-sublinear fractional maximal{\
operators}, respectively. The author also studies two-weighted weak type
estimate on product $L^{p}\left( {\mathbb{R}^{n}}\right) $ for rough
multi-sublinear fractional maximal{\ operators. }In fact, this article is
the rough kernel versions of \cite{Kenig, Shi}'s results.
\end{abstract}

\maketitle

\section{Introduction}

Let ${\mathbb{R}^{n}}$ be the $n$-dimensional Euclidean space of points $%
x=(x_{1},...,x_{n})$ with norm $|x|=\left( \dsum
\limits_{i=1}^{n}x_{i}^{2}\right) ^{\frac{1}{2}}$ and $\left( {\mathbb{R}^{n}%
}\right) ^{m}={\mathbb{R}^{n}\times \ldots \times \mathbb{R}^{n}}$ be the $m$%
-fold product spaces $\left( m\in 
\mathbb{N}
\right) $. Throughout this paper, we denote by $\overrightarrow{y}=\left(
y_{1},\ldots ,y_{m}\right) $ and $x,y_{1},\ldots ,y_{m}\in $ ${\mathbb{R}^{n}%
}$, $d\overrightarrow{y}=dy_{1}\ldots dy_{m}$, and by $\overrightarrow{f}$
the $m$-tuple $\left( f_{1},...,f_{m}\right) $, $m$, $n$ the nonnegative
integers with $n\geq 2$, $m\geq 1$. Let also $S^{mn-1}$ denote the unit
sphere of ${\mathbb{R}^{mn}}$ only with the condition $\Omega \in
L^{s}\left( S^{mn-1}\right) $ for some $s>1$.

It is well known that, for the purpose of researching non-smoothness partial
differential equation, mathematicians pay more attention to the singular
integrals. Moreover, the classical fractional integral operator(Riesz
potential) $I_{\alpha }$ plays important roles in many fields of
mathematics. For example, its most significant property is that $I_{\alpha }$
maps $L^{p}\left( {\mathbb{R}^{n}}\right) $ continuously into $L^{q}\left( {%
\mathbb{R}^{n}}\right) $, with $\frac{1}{q}=\frac{1}{p}-\frac{\alpha }{n}$
and $1<p<\frac{n}{\alpha }$, through the well known Hardy-Littlewood-Sobolev
imbedding theorem (see \cite{St}). On the other hand, the theory of
multilinear analysis was received extensive studies in the last 3 decades.
Among numerous references, in the following we list a few of them about
multilinear maximal function and multilinear fractional integral which are
related to the study in this article.

In 1992, Grafakos \cite{Grafakos1} first studied multilinear maximal
function and multilinear fractional integral defined by 
\begin{equation*}
M_{\alpha ,\overrightarrow{\theta }}^{\left( m\right) }\left( 
\overrightarrow{f}\right) \left( x\right) =\sup_{t>0}\frac{1}{r^{n-\alpha }}%
\int \limits_{\left \vert y\right \vert <r}\left \vert \dprod
\limits_{i=1}^{m}f_{i}\left( x-\theta _{i}y\right) \right \vert dy
\end{equation*}%
and 
\begin{equation*}
I_{\alpha ,\overrightarrow{\theta }}^{\left( m\right) }\left( 
\overrightarrow{f}\right) \left( x\right) =\int \limits_{{\mathbb{R}^{n}}}%
\frac{1}{\left \vert y\right \vert ^{n-\alpha }}\dprod
\limits_{i=1}^{m}f_{i}\left( x-\theta _{i}y\right) dy,
\end{equation*}%
where $\overrightarrow{f}=\left( f_{1},...,f_{m}\right) $, $\overrightarrow{%
\theta }=\left( \theta _{1},...,\theta _{m}\right) $ is a fixed vector with
distinct nonzero real numbers and $0<\alpha <n$. We note that, if we simply
take $m=1$ and $\theta _{i}=1$, then $M_{\alpha }$ and $I_{\alpha }$ are
just the operators studied by Muckenhoupt and Wheeden in \cite{Muckenhoupt}.

In 1999, another multilinear fractional integral was defined by Kenig and
Stein \cite{Kenig} as follows:

\begin{equation*}
I_{\alpha }^{\left( m\right) }\left( \overrightarrow{f}\right) \left(
x\right) =\dint \limits_{\left( {\mathbb{R}^{n}}\right) ^{m}}\frac{1}{\left
\vert \overrightarrow{y}\right \vert ^{mn-\alpha }}\dprod%
\limits_{i=1}^{m}f_{i}\left( x-y_{i}\right) d\overrightarrow{y},
\end{equation*}%
where $\left \vert \overrightarrow{y}\right \vert =\left \vert y_{1}\right
\vert +\cdots +\left \vert y_{m}\right \vert $. They proved that $I_{\alpha
}^{\left( m\right) }$ is of strong type $\left( L^{p_{1}}\times
L^{p_{2}}\times \cdots \times L^{p_{m}},L^{q}\right) $ and weak type $\left(
L^{p_{1}}\times L^{p_{2}}\times \cdots \times L^{p_{m}},L^{q,\infty }\right) 
$. And also corresponding multi-sublinear fractional maximal{\ operator is
defined by}%
\begin{equation*}
M_{\alpha }^{\left( m\right) }\left( \overrightarrow{f}\right) \left(
x\right) =\sup_{r>0}\frac{1}{r^{mn-\alpha }}\int \limits_{\left \vert 
\overrightarrow{y}\right \vert <r}\dprod \limits_{i=1}^{m}\left \vert
f_{i}\left( x-y_{i}\right) \right \vert d\overrightarrow{y},
\end{equation*}%
where $0<\alpha <mn$ and $x\in $ ${\mathbb{R}^{n}}$.

In 2008, Shi and Tao built the one-weighted and two-weighted boundedness on
product $L^{p}\left( {\mathbb{R}^{n}}\right) $ space for $I_{\alpha
}^{\left( m\right) }$ and these authors also considered two-weighted weak
type estimate on product $L^{p}\left( {\mathbb{R}^{n}}\right) $ for $%
M_{\alpha }^{\left( m\right) }$.

Motivated by \cite{Kenig, Shi}, in this paper we will introduce the
following rough multilinear fractional operators $I_{\Omega ,\alpha
}^{\left( m\right) }$ and rough multi-sublinear fractional maximal operators 
$M_{\Omega ,\alpha }^{\left( m\right) }$, which are the more generalizations
of the classical setting and study them on product spaces $L^{p_{1}}\left( {%
\mathbb{R}^{n}}\right) \times L^{p_{2}}\left( {\mathbb{R}^{n}}\right) \times
\cdots \times L^{p_{m}}\left( {\mathbb{R}^{n}}\right) $%
\begin{equation*}
I_{\Omega ,\alpha }^{\left( m\right) }\left( \overrightarrow{f}\right)
\left( x\right) =\dint \limits_{\left( {\mathbb{R}^{n}}\right) ^{m}}\frac{%
\Omega \left( \overrightarrow{y}\right) }{\left \vert \overrightarrow{y}%
\right \vert ^{mn-\alpha }}\dprod \limits_{i=1}^{m}f_{i}\left(
x-y_{i}\right) d\overrightarrow{y},
\end{equation*}%
\begin{equation*}
M_{\Omega ,\alpha }^{\left( m\right) }\left( \overrightarrow{f}\right)
\left( x\right) =\sup_{r>0}\frac{1}{r^{mn-\alpha }}\dint \limits_{\left
\vert \overrightarrow{y}\right \vert <r}\left \vert \Omega \left( 
\overrightarrow{y}\right) \right \vert \dprod \limits_{i=1}^{m}\left \vert
f_{i}\left( x-y_{i}\right) \right \vert d\overrightarrow{y},
\end{equation*}%
where $\left \vert \overrightarrow{y}\right \vert =\left \vert
y_{1}\right
\vert +\cdots +\left \vert y_{m}\right \vert $.

At last, it will be an interesting question whether Kenig and Stein's famous
result (Theorem 1 in \cite{Kenig}) and Shi and Tao's weighted conclusions
can be extended to the operators $I_{\Omega ,\alpha }^{\left( m\right) }$
and $M_{\Omega ,\alpha }^{\left( m\right) }$ for $m>1$ and non-smooth kernel 
$\Omega $.

In this paper, we will give the positive answers to this question, and
simultaneity extend Kenig and Stein's famous result (Theorem 1 in \cite%
{Kenig}) and Shi and Tao's weighted conclusions to the context for{\ }the
operators $I_{\Omega ,\alpha }^{\left( m\right) }$ and{\ }$M_{\Omega ,\alpha
}^{\left( m\right) }$, and also show their weighted boundedness,
respectively.

Throughout this paper, the letter $C$ will denote a constant whose value may
vary at each occurrence, but it is independent of the main parameters.

\section{definitions and main results}

We first recall the definition of weighted Lebesgue spaces. By a "weight" we
will mean a non-negative function $w$ that is a positive measure a.e. on ${%
\mathbb{R}^{n}}$.

\begin{definition}
$\left( \text{\textbf{Weighted Lebesgue space}}\right) $ Let $1\leq p\leq
\infty $ and given a weight $w\left( x\right) \in A_{p}\left( {{\mathbb{R}%
^{n}}}\right) $, we shall define weighted Lebesgue spaces as 
\begin{eqnarray*}
L_{p}(w) &\equiv &L_{p}({{\mathbb{R}^{n}}},w)=\left \{ f:\Vert f\Vert
_{L_{p,w}}=\left( \dint \limits_{{{\mathbb{R}^{n}}}}|f(x)|^{p}w(x)dx\right)
^{\frac{1}{p}}<\infty \right \} ,\qquad 1\leq p<\infty . \\
L_{\infty ,w} &\equiv &L_{\infty }({{\mathbb{R}^{n}}},w)=\left \{ f:\Vert
f\Vert _{L_{\infty ,w}}=\limfunc{esssup}\limits_{x\in {\mathbb{R}^{n}}%
}|f(x)|w(x)<\infty \right \} .
\end{eqnarray*}
\end{definition}

Before showing our main results, we next recall the definitions of some
relative weight classes. In the following definitions, the function $w$ and
the function pair $\left( u,\upsilon \right) $ are all locally integrable
nonnegative functions. Moreover, $C>0$ and $Q$ denotes a cube in ${{\mathbb{R%
}^{n}}}$ with its sides parallel to the coordinate axes.

\begin{definition}
$\left( \text{\textbf{Class of }}A_{p}\right) $ A function $w$ is said to
belong to $A_{p}\left( 1<p<\infty \right) $ if 
\begin{equation}
\sup \limits_{Q\subset {{\mathbb{R}^{n}}}}\left( \frac{1}{|Q|}\dint
\limits_{Q}w(x)dx\right) \left( \frac{1}{|Q|}\dint
\limits_{Q}w(x)^{1-p^{\prime }}dx\right) ^{p-1}\leq C,  \label{100}
\end{equation}%
where $p^{\prime }=\frac{p}{p-1}$. The condition (\ref{100}) is called the $%
A_{p}$-condition, and the weights which satisfy it are called $A_{p}$%
-weights. The property of the $A_{p}$-weights implies that generally
speaking, we should check whether a weight $w$ satisfies an $A_{p}$%
-condition or not.
\end{definition}

\begin{definition}
\cite{Muckenhoupt}$\left( \text{\textbf{Class of} }A_{p,q}\right) $ A
function $w$ is said to belong to the Muckenhoupt-Wheeden class $A\left(
p,q\right) \left( 1<p<q<\infty \right) $ if%
\begin{equation*}
\sup \limits_{Q\subset {{\mathbb{R}^{n}}}}\left( \frac{1}{|Q|}\dint
\limits_{Q}w(x)^{q}dx\right) ^{\frac{1}{q}}\left( \frac{1}{|Q|}\dint
\limits_{Q}w(x)^{-p^{\prime }}dx\right) ^{\frac{1}{p^{\prime }}}\leq C.
\end{equation*}
\end{definition}

\begin{definition}
\cite{Shi}$\left( \text{\textbf{Class of} }A_{p,q}^{\alpha }\right) $ A
function pair $\left( u,\upsilon \right) $ is said to belong to the radial
Muckenhoupt-Wheeden class $A_{p,q}^{\alpha }$ $\left( 1\leq p\leq q<\infty 
\text{ and }0\leq \alpha <n\right) $ if%
\begin{equation*}
\left \vert Q\right \vert ^{\frac{1}{q}+\frac{\alpha }{n}-\frac{1}{p}}\left( 
\frac{1}{\left \vert Q\right \vert }\dint \limits_{Q}u\left( x\right)
dx\right) ^{\frac{1}{q}}\left( \frac{1}{\left \vert Q\right \vert }\dint
\limits_{Q}\upsilon \left( x\right) ^{\left( 1-p^{\prime }\right) }dx\right)
^{\frac{1}{p^{\prime }}}\leq C,
\end{equation*}%
when $1<p<\infty $; and%
\begin{equation*}
\left \vert Q\right \vert ^{\frac{1}{q}+\frac{\alpha }{n}-1}\left( \frac{1}{%
\left \vert Q\right \vert }\dint \limits_{Q}u\left( x\right) dx\right) ^{%
\frac{1}{q}}\leq C\upsilon \left( x\right) \qquad \text{a.e. }x\in Q,
\end{equation*}%
when $p=1$. Recall the definition of $A_{p}$ weight, it is easy to see that $%
\left( u,\upsilon \right) \in A_{p}$ if and only if $\left( u,\upsilon
\right) \in A_{p,p}^{0}$ for $1\leq p<\infty $.
\end{definition}

In this paper, we prove the following results.

\begin{theorem}
\label{teo1}Let $0<\alpha <mn$, $1\leq s^{\prime }<\frac{mn}{\alpha }$, and $%
\Omega $ be homogeneous of degree zero on ${\mathbb{R}^{mn}}$ with $\Omega
\in L^{s}\left( S^{mn-1}\right) \left( s>1\right) $, $\frac{1}{s}+\frac{1}{%
s^{\prime }}=1$. Let $\frac{1}{p}=\frac{1}{p_{1}}+\frac{1}{p_{2}}+\cdots +%
\frac{1}{p_{m}}-\frac{\alpha }{n}>0$.

$\left( i\right) $ If each $s^{\prime }<p_{i}$, then there exists a constant 
$C>0$ such that%
\begin{equation}
\left \Vert M_{\Omega ,\alpha }^{\left( m\right) }\left( \overrightarrow{f}%
\right) \right \Vert _{L^{p}\left( {\mathbb{R}^{n}}\right) }\leq C\dprod
\limits_{i=1}^{m}\left \Vert f_{i}\right \Vert _{L^{p_{i}}\left( {\mathbb{R}%
^{n}}\right) }.  \label{1}
\end{equation}

$\left( ii\right) $ If $p_{i}=s^{\prime }$ for some $i$, then there exists a
constant $C>0$ such that%
\begin{equation}
\left \Vert M_{\Omega ,\alpha }^{\left( m\right) }\left( \overrightarrow{f}%
\right) \right \Vert _{L^{p,\infty }\left( {\mathbb{R}^{n}}\right) }\leq
C\dprod \limits_{i=1}^{m}\left \Vert f_{i}\right \Vert _{L^{p_{i}}\left( {%
\mathbb{R}^{n}}\right) }.  \label{2}
\end{equation}
\end{theorem}

\begin{theorem}
\label{teo2}Suppose the same conditions and notations of that in Theorem \ref%
{teo1},

$\left( i\right) $ if each $s^{\prime }<p_{i}$, then there exists a constant 
$C>0$ such that%
\begin{equation}
\left \Vert I_{\Omega ,\alpha }^{\left( m\right) }\left( \overrightarrow{f}%
\right) \right \Vert _{L^{p}\left( {\mathbb{R}^{n}}\right) }\leq C\dprod
\limits_{i=1}^{m}\left \Vert f_{i}\right \Vert _{L^{p_{i}}\left( {\mathbb{R}%
^{n}}\right) };  \label{3}
\end{equation}

$\left( ii\right) $ if $p_{i}=s^{\prime }$ for some $i$, then there exists a
constant $C>0$ such that%
\begin{equation}
\left \Vert I_{\Omega ,\alpha }^{\left( m\right) }\left( \overrightarrow{f}%
\right) \right \Vert _{L^{p,\infty }\left( {\mathbb{R}^{n}}\right) }\leq
C\dprod \limits_{i=1}^{m}\left \Vert f_{i}\right \Vert _{L^{p_{i}}\left( {%
\mathbb{R}^{n}}\right) }.  \label{4}
\end{equation}
\end{theorem}

\begin{theorem}
\label{teo3}Let $0<\alpha <mn$, $1\leq s^{\prime }<\frac{mn}{\alpha }$, and $%
\Omega $ be homogeneous of degree zero on ${\mathbb{R}^{mn}}$ with $\Omega
\in L^{s}\left( S^{mn-1}\right) \left( s>1\right) $, $\frac{1}{s}+\frac{1}{%
s^{\prime }}=1$. Suppose that $f_{i}\in L_{w^{p_{i}}}^{p_{i}}\left( {\mathbb{%
R}^{n}}\right) $ with $s^{\prime }<p_{i}<\frac{mn}{\alpha }\left(
i=1,2,\ldots ,m\right) $ and $w\left( x\right) ^{s^{\prime }}\in \dbigcap
\limits_{i=1}^{n}A_{\frac{p}{s^{\prime }},\frac{q}{s^{\prime }}}$, where $%
\frac{1}{q_{i}}=\frac{1}{p_{i}}-\frac{\alpha }{mn}$. If let $\frac{1}{p}=%
\frac{1}{p_{1}}+\frac{1}{p_{2}}+\cdots +\frac{1}{p_{m}}-\frac{\alpha }{n}$,
then there is a constant $C>0$, independent of $f_{i}$, such that 
\begin{equation}
\left \Vert M_{\Omega ,\alpha }^{\left( m\right) }\left( \overrightarrow{f}%
\right) \right \Vert _{L_{w^{p}}^{p}\left( {\mathbb{R}^{n}}\right) }\leq
C\dprod \limits_{i=1}^{m}\left \Vert f_{i}\right \Vert
_{L_{w^{p_{i}}}^{p_{i}}\left( {\mathbb{R}^{n}}\right) }.  \label{5}
\end{equation}
\end{theorem}

\begin{theorem}
\label{teo4}Suppose the same conditions and notations of that in Theorem \ref%
{teo3}, then there is a constant $C>0$, independent of $f_{i}$, such that%
\begin{equation}
\left \Vert I_{\Omega ,\alpha }^{\left( m\right) }\left( \overrightarrow{f}%
\right) \right \Vert _{L_{w^{p}}^{p}\left( {\mathbb{R}^{n}}\right) }\leq
C\dprod \limits_{i=1}^{m}\left \Vert f_{i}\right \Vert
_{L_{w^{p_{i}}}^{p_{i}}\left( {\mathbb{R}^{n}}\right) }.  \label{6}
\end{equation}
\end{theorem}

\begin{theorem}
\label{teo5}Let $0<\alpha <mn$, $1\leq s^{\prime }<\frac{mn}{\alpha }$, and $%
\Omega $ be homogeneous of degree zero on ${\mathbb{R}^{mn}}$ with $\Omega
\in L^{s}\left( S^{mn-1}\right) \left( s>1\right) $, $\frac{1}{s}+\frac{1}{%
s^{\prime }}=1$. Assume that $\left( u,\upsilon \right) $ is a pair of
weights, $s^{\prime }<p_{i}<q_{i}<\infty $, for each $i=1,2,\ldots ,m$. Let $%
\frac{1}{p}=\frac{1}{q_{1}}+\frac{1}{q_{2}}+\cdots +\frac{1}{q_{m}}$, $\frac{%
1}{p_{i}}+\frac{1}{p_{i}^{\prime }}=1$. If there exists $r_{i}>1$ such that,
for every cube $Q$ in ${\mathbb{R}^{n}}$,%
\begin{equation}
\left \vert Q\right \vert ^{\frac{s^{\prime }}{q_{i}}+\frac{\alpha s^{\prime
}}{mn}-\frac{s^{\prime }}{p_{i}}}\left( \frac{1}{\left \vert Q\right \vert }%
\dint \limits_{Q}u\left( x\right) dx\right) ^{\frac{s^{\prime }}{q_{i}}%
}\left( \frac{1}{\left \vert Q\right \vert }\dint \limits_{Q}\upsilon \left(
x\right) ^{r_{i}\left( 1-\left( \frac{p_{i}}{s^{\prime }}\right) ^{\prime
}\right) }dx\right) ^{\frac{1}{r_{i}\left( \frac{p_{i}}{s^{\prime }}\right)
^{\prime }}}\leq C,  \label{7}
\end{equation}%
then for arbitrary $f_{i}\in L_{\upsilon }^{p_{i}}\left( {\mathbb{R}^{n}}%
\right) $, there is a constant $C>0$, independent of $f_{i}$, such that%
\begin{equation}
\left \Vert M_{\Omega ,\alpha }^{\left( m\right) }\left( \overrightarrow{f}%
\right) \right \Vert _{L_{u}^{p}\left( {\mathbb{R}^{n}}\right) }\leq C\dprod
\limits_{i=1}^{m}\left \Vert f_{i}\right \Vert _{L_{\upsilon }^{p_{i}}\left( 
{\mathbb{R}^{n}}\right) }.  \label{8}
\end{equation}
\end{theorem}

\begin{theorem}
\label{teo6}Let $0<\alpha <mn$, $1\leq s^{\prime }<\frac{mn}{\alpha }$, and $%
\Omega $ be homogeneous of degree zero on ${\mathbb{R}^{mn}}$ with $\Omega
\in L^{s}\left( S^{mn-1}\right) \left( s>1\right) $, $\frac{1}{s}+\frac{1}{%
s^{\prime }}=1$, $\left( u,\upsilon \right) $ is a pair of weights. If for
every $i=1,2,\ldots ,m$, $s^{\prime }<p_{i}<mp<\infty $ and there exists $%
r_{i}>1$ such that for every cube $Q$ in ${\mathbb{R}^{n}}$, such that%
\begin{equation*}
\left \vert Q\right \vert ^{\frac{s^{\prime }}{mp}+\frac{\alpha s^{\prime }}{%
mn}-\frac{s^{\prime }}{p_{i}}}\left( \frac{1}{\left \vert Q\right \vert }%
\dint \limits_{Q}u\left( x\right) ^{r_{i}}dx\right) ^{\frac{s^{\prime }}{%
r_{i}mp}}\left( \frac{1}{\left \vert Q\right \vert }\dint
\limits_{Q}\upsilon \left( x\right) ^{r_{i}\left( 1-\left( \frac{p_{i}}{%
s^{\prime }}\right) ^{\prime }\right) }dx\right) ^{\frac{1}{r_{i}\left( 
\frac{p_{i}}{s^{\prime }}\right) ^{\prime }}}\leq C,
\end{equation*}%
then for arbitrary $f_{i}\in L_{\upsilon }^{p_{i}}\left( {\mathbb{R}^{n}}%
\right) $, there is a constant $C>0$, independent of $f_{i}$, such that%
\begin{equation}
\left \Vert I_{\Omega ,\alpha }^{\left( m\right) }\left( \overrightarrow{f}%
\right) \right \Vert _{L_{u}^{p}\left( {\mathbb{R}^{n}}\right) }\leq C\dprod
\limits_{i=1}^{m}\left \Vert f_{i}\right \Vert _{L_{\upsilon }^{p_{i}}\left( 
{\mathbb{R}^{n}}\right) }.  \label{9}
\end{equation}
\end{theorem}

\begin{remark}
Theorem \ref{teo2} implies the well-known Hardy-Littlewood-Sobolev Theorem 
\cite{Stein93}, i.e. the case $m=1$, $\Omega \equiv 1$ and $s=\infty $.
Theorem \ref{teo2} implies Theorem 1 in \cite{Kenig} when $\Omega \equiv 1$
and $s=\infty $. One can also obtain from (\ref{4}) of Theorem \ref{teo2}
that the operator $I_{\Omega ,\alpha }$ is weak type $\left( 1,\frac{n}{%
n-\alpha }\right) $. Theorem \ref{teo4} extends the weighted boundedness of $%
I_{\Omega ,\alpha }$ in \cite{Ding}. If $m=1$, $\Omega \equiv 1$ and $%
s=\infty $, Theorem \ref{6} becomes the results of Garc\'{\i}a-Cuerva and
Martell \cite{Garcia}, where the two weighted boundedness of $I_{\alpha }$
is considered.
\end{remark}

\begin{theorem}
\label{teo53}Let $0<\alpha <mn$, $1\leq s^{\prime }<\frac{mn}{\alpha }$, and 
$\Omega $ be homogeneous of degree zero on ${\mathbb{R}^{mn}}$ with $\Omega
\in L^{s}\left( S^{mn-1}\right) \left( s>1\right) $, $\frac{1}{s}+\frac{1}{%
s^{\prime }}=1$. Assume that $\left( u,\upsilon \right) $ is a pair of
weights, $s^{\prime }\leq p_{i}\leq q_{i}<\infty $ for each $i=1,2,\ldots ,m$%
. $\frac{1}{p}=\frac{1}{q_{1}}+\frac{1}{q_{2}}+\cdots +\frac{1}{q_{m}}$, $%
\frac{1}{p_{i}}+\frac{1}{p_{i}^{\prime }}=1$. If $\left( u,\upsilon \right)
\in \dbigcap \limits_{i=1}^{m}A_{\frac{p_{i}}{s^{\prime }},\frac{q_{i}}{%
s^{\prime }}}^{\frac{\alpha s^{\prime }}{m}}$, then for every $f_{i}\in
L_{\upsilon }^{p_{i}}\left( {\mathbb{R}^{n}}\right) $, there is a constant $%
C $, independent of $f_{i}$, such that%
\begin{equation}
\left \Vert M_{\Omega ,\alpha }^{\left( m\right) }\left( \overrightarrow{f}%
\right) \right \Vert _{L_{u}^{p,\infty }\left( {\mathbb{R}^{n}}\right) }\leq
C\dprod \limits_{i=1}^{m}\left \Vert f_{i}\right \Vert _{L_{\upsilon
}^{p_{i}}\left( {\mathbb{R}^{n}}\right) }.  \label{52}
\end{equation}
\end{theorem}

\begin{remark}
We note that if a pair weights $\left( u,\upsilon \right) $ satisfies (\ref%
{7}) with $r_{i}>1$, then $\left( u,\upsilon \right) \in \dbigcap
\limits_{i=1}^{m}A_{\frac{p_{i}}{s^{\prime }},\frac{q_{i}}{s^{\prime }}}^{%
\frac{\alpha s^{\prime }}{m}}$. This result says that, by $A_{\frac{p_{i}}{%
s^{\prime }},\frac{q_{i}}{s^{\prime }}}^{\frac{\alpha s^{\prime }}{m}} $,
which is weaker than the condition (\ref{7}), the operator $M_{\Omega
,\alpha }^{\left( m\right) }$ turns out to be of weak type.
\end{remark}

\section{Product $L^{p}$-estimates for $I_{\Omega ,\protect \alpha }^{\left(
m\right) }$ and $M_{\Omega ,\protect \alpha }^{\left( m\right) }$}

To prove Theorems \ref{teo1} and \ref{teo2}, we need some lemmas.

\begin{lemma}
\label{lemma1*}Let $0<\alpha <mn$, $1\leq s^{\prime }<\frac{mn}{\alpha }$,
and $\Omega $ be homogeneous of degree zero on ${\mathbb{R}^{mn}}$ with $%
\Omega \in L^{s}\left( S^{mn-1}\right) \left( s>1\right) $, $\frac{1}{s}+%
\frac{1}{s^{\prime }}=1$, assume that the function $f_{i}\in L^{p_{i}}\left( 
{\mathbb{R}^{n}}\right) $ with $1\leq p_{i}\leq \infty \left( i=1,2,\ldots
,m\right) $, then there exists a constant $C>0$ such that for any $x\in {%
\mathbb{R}^{n}}$,%
\begin{equation}
M_{\Omega ,\alpha }^{\left( m\right) }\left( \overrightarrow{f}\right)
\left( x\right) \leq C\left[ M_{\alpha s^{\prime }}^{\left( m\right) }\left(
\left \vert f_{1}\right \vert ^{s^{\prime }},\left \vert f_{2}\right \vert
^{s^{\prime }},\ldots ,\left \vert f_{m}\right \vert ^{s^{\prime }}\right)
\left( x\right) \right] ^{\frac{1}{s^{\prime }}}.  \label{1*}
\end{equation}
\end{lemma}

\begin{proof}
By $s>1$, $\Omega \in L^{s}\left( S^{mn-1}\right) $ and the H\"{o}lder's
inequality, we have%
\begin{eqnarray*}
&&\frac{1}{r^{mn-\alpha }}\dint \limits_{\left \vert \overrightarrow{y}%
\right \vert <r}\left \vert \Omega \left( \overrightarrow{y}\right) \right
\vert \dprod \limits_{i=1}^{m}\left \vert f_{i}\left( x-y_{i}\right) \right
\vert d\overrightarrow{y} \\
&\leq &\frac{1}{r^{mn-\alpha }}\left( \dint \limits_{\left \vert 
\overrightarrow{y}\right \vert <r}\dprod \limits_{i=1}^{m}\left \vert
f_{i}\left( x-y_{i}\right) \right \vert ^{s^{\prime }}d\overrightarrow{y}%
\right) ^{\frac{1}{s^{\prime }}}\left( \dint \limits_{\left \vert 
\overrightarrow{y}\right \vert <r}\left \vert \Omega \left( \overrightarrow{y%
}\right) \right \vert ^{s}d\overrightarrow{y}\right) ^{\frac{1}{s}} \\
&\leq &C\sup_{r>0}\frac{1}{r^{mn\left( 1-\frac{1}{s}\right) -\alpha }}\left(
\dint \limits_{\left \vert \overrightarrow{y}\right \vert <r}\dprod
\limits_{i=1}^{m}\left \vert f_{i}\left( x-y_{i}\right) \right \vert
^{s^{\prime }}d\overrightarrow{y}\right) ^{\frac{1}{s^{\prime }}} \\
&\leq &C\sup_{r>0}\left( \frac{1}{r^{mn-\alpha s^{\prime }}}\dint
\limits_{\left \vert \overrightarrow{y}\right \vert <r}\dprod
\limits_{i=1}^{m}\left \vert f_{i}\left( x-y_{i}\right) \right \vert
^{s^{\prime }}d\overrightarrow{y}\right) ^{\frac{1}{s^{\prime }}} \\
&\leq &C\left( \sup_{r>0}\frac{1}{r^{mn-\alpha s^{\prime }}}\dint
\limits_{\left \vert \overrightarrow{y}\right \vert <r}\dprod
\limits_{i=1}^{m}\left \vert f_{i}\left( x-y_{i}\right) \right \vert
^{s^{\prime }}d\overrightarrow{y}\right) ^{\frac{1}{s^{\prime }}} \\
&\leq &C\left[ M_{\alpha s^{\prime }}^{\left( m\right) }\left( \left \vert
f_{1}\right \vert ^{s^{\prime }},\left \vert f_{2}\right \vert ^{s^{\prime
}},\ldots ,\left \vert f_{m}\right \vert ^{s^{\prime }}\right) \left(
x\right) \right] ^{\frac{1}{s^{\prime }}}.
\end{eqnarray*}%
This completes the proof of the Lemma \ref{lemma1*}.
\end{proof}

The following Lemma \ref{lemma2*} plays a key role in the manuscript and the
idea of the proof of Lemma \ref{lemma2*} belongs to Welland \cite{Welland}.

\begin{lemma}
\label{lemma2*}$\left( \text{\textbf{Welland's inequality for }}I_{\Omega
,\alpha }^{\left( m\right) }\text{ \textbf{and} }M_{\Omega ,\alpha }^{\left(
m\right) }\right) $Let $0<\alpha <mn$, and let $f_{i}\in L^{p_{i}}\left( {%
\mathbb{R}^{n}}\right) $ with $1\leq p_{i}\leq \infty \left( i=1,2,\ldots
,m\right) $. For any $0<\epsilon <\min \left \{ \alpha ,mn-\alpha \right \} $%
, there exists a constant $C>0$ such that for any $x\in {\mathbb{R}^{n}}$,%
\begin{equation}
\left \vert I_{\Omega ,\alpha }^{\left( m\right) }\left( \overrightarrow{f}%
\right) \left( x\right) \right \vert \leq C\left[ M_{\Omega ,\alpha
+\epsilon }^{\left( m\right) }\left( \overrightarrow{f}\right) \left(
x\right) \right] ^{\frac{1}{2}}\left[ M_{\Omega ,\alpha -\epsilon }^{\left(
m\right) }\left( \overrightarrow{f}\right) \left( x\right) \right] ^{\frac{1%
}{2}}.  \label{2*}
\end{equation}
\end{lemma}

\begin{proof}
Fix $x\in {\mathbb{R}^{n}}$ and $0<\epsilon <\min \left \{ \alpha ,mn-\alpha
\right \} $, for any $\delta >0$ we decompose as follow,%
\begin{eqnarray*}
\left \vert I_{\Omega ,\alpha }^{\left( m\right) }\left( \overrightarrow{f}%
\right) \left( x\right) \right \vert &\leq &\dint \limits_{\left( {\mathbb{R}%
^{n}}\right) ^{m}}\frac{\left \vert \Omega \left( \overrightarrow{y}\right)
\right \vert }{\left \vert \overrightarrow{y}\right \vert ^{mn-\alpha }}%
\dprod \limits_{i=1}^{m}\left \vert f_{i}\left( x-y_{i}\right) \right \vert d%
\overrightarrow{y} \\
&\leq &\dint \limits_{\left \vert \overrightarrow{y}\right \vert <\delta }%
\frac{\left \vert \Omega \left( \overrightarrow{y}\right) \right \vert }{%
\left \vert \overrightarrow{y}\right \vert ^{mn-\alpha }}\dprod
\limits_{i=1}^{m}\left \vert f_{i}\left( x-y_{i}\right) \right \vert d%
\overrightarrow{y}+\dint \limits_{\left \vert \overrightarrow{y}\right \vert
>\delta }\frac{\left \vert \Omega \left( \overrightarrow{y}\right) \right
\vert }{\left \vert \overrightarrow{y}\right \vert ^{mn-\alpha }}\dprod
\limits_{i=1}^{m}\left \vert f_{i}\left( x-y_{i}\right) \right \vert d%
\overrightarrow{y} \\
&=&:J_{1}+J_{2}.
\end{eqnarray*}

Now, we will estimate $J_{1}$ and $J_{2}$, respectively. For $J_{1}$,%
\begin{eqnarray*}
J_{1} &=&\dsum \limits_{j=0}^{\infty }\dint \limits_{\overrightarrow{y}\in
B\left( 2^{-j}\delta \right) \diagdown B\left( 2^{-j-1}\delta \right) }\frac{%
\left \vert \Omega \left( \overrightarrow{y}\right) \right \vert }{\left
\vert \overrightarrow{y}\right \vert ^{mn-\alpha }}\dprod
\limits_{i=1}^{m}\left \vert f_{i}\left( x-y_{i}\right) \right \vert d%
\overrightarrow{y} \\
&\leq &\dsum \limits_{j=0}^{\infty }\frac{1}{\left( 2^{-j-1}\delta \right)
^{mn-\alpha }}\dint \limits_{\overrightarrow{y}\in B\left( 2^{-j}\delta
\right) \diagdown B\left( 2^{-j-1}\delta \right) }\left \vert \Omega \left( 
\overrightarrow{y}\right) \right \vert \dprod \limits_{i=1}^{m}\left \vert
f_{i}\left( x-y_{i}\right) \right \vert d\overrightarrow{y} \\
&\leq &C\dsum \limits_{j=0}^{\infty }\frac{1}{\left( 2^{-j}\delta \right)
^{mn-\alpha }}\dint \limits_{\overrightarrow{y}\in B\left( 2^{-j}\delta
\right) }\left \vert \Omega \left( \overrightarrow{y}\right) \right \vert
\dprod \limits_{i=1}^{m}\left \vert f_{i}\left( x-y_{i}\right) \right \vert d%
\overrightarrow{y} \\
&\leq &C\dsum \limits_{j=0}^{\infty }\frac{\left( 2^{-j}\delta \right)
^{\epsilon }}{\left( 2^{-j}\delta \right) ^{mn-\alpha +\epsilon }}\dint
\limits_{\overrightarrow{y}\in B\left( 2^{-j}\delta \right) }\left \vert
\Omega \left( \overrightarrow{y}\right) \right \vert \dprod
\limits_{i=1}^{m}\left \vert f_{i}\left( x-y_{i}\right) \right \vert d%
\overrightarrow{y} \\
&\leq &C\delta ^{\epsilon }\dsum \limits_{j=0}^{\infty }\left( 2^{-j\epsilon
}\right) M_{\Omega ,\alpha -\epsilon }^{\left( m\right) }\left( 
\overrightarrow{f}\right) \left( x\right) \\
&\leq &C\delta ^{\epsilon }M_{\Omega ,\alpha -\epsilon }^{\left( m\right)
}\left( \overrightarrow{f}\right) \left( x\right) .
\end{eqnarray*}

As for $J_{2}$, we have%
\begin{eqnarray*}
J_{2} &=&\dsum \limits_{j=0}^{\infty }\dint \limits_{\overrightarrow{y}\in
B\left( 2^{j+1}\delta \right) \diagdown B\left( 2^{j}\delta \right) }\frac{%
\left \vert \Omega \left( \overrightarrow{y}\right) \right \vert }{\left
\vert \overrightarrow{y}\right \vert ^{mn-\alpha }}\dprod
\limits_{i=1}^{m}\left \vert f_{i}\left( x-y_{i}\right) \right \vert d%
\overrightarrow{y} \\
&\leq &\dsum \limits_{j=0}^{\infty }\frac{1}{\left( 2^{j}\delta \right)
^{mn-\alpha }}\dint \limits_{\overrightarrow{y}\in B\left( 2^{j+1}\delta
\right) \diagdown B\left( 2^{j}\delta \right) }\left \vert \Omega \left( 
\overrightarrow{y}\right) \right \vert \dprod \limits_{i=1}^{m}\left \vert
f_{i}\left( x-y_{i}\right) \right \vert d\overrightarrow{y} \\
&\leq &C\dsum \limits_{j=0}^{\infty }\frac{1}{\left( 2^{j}\delta \right)
^{mn-\alpha }}\dint \limits_{\overrightarrow{y}\in B\left( 2^{j+1}\delta
\right) }\left \vert \Omega \left( \overrightarrow{y}\right) \right \vert
\dprod \limits_{i=1}^{m}\left \vert f_{i}\left( x-y_{i}\right) \right \vert d%
\overrightarrow{y} \\
&\leq &C\dsum \limits_{j=0}^{\infty }\frac{\left( 2^{j}\delta \right)
^{\epsilon }}{\left( 2^{j}\delta \right) ^{mn-\alpha -\epsilon }}\dint
\limits_{\overrightarrow{y}\in B\left( 2^{j+1}\delta \right) }\left \vert
\Omega \left( \overrightarrow{y}\right) \right \vert \dprod
\limits_{i=1}^{m}\left \vert f_{i}\left( x-y_{i}\right) \right \vert d%
\overrightarrow{y} \\
&\leq &C\delta ^{-\epsilon }\dsum \limits_{j=0}^{\infty }\left(
2^{-j\epsilon }\right) M_{\Omega ,\alpha +\epsilon }^{\left( m\right)
}\left( \overrightarrow{f}\right) \left( x\right) \\
&\leq &C\delta ^{-\epsilon }M_{\Omega ,\alpha +\epsilon }^{\left( m\right)
}\left( \overrightarrow{f}\right) \left( x\right) .
\end{eqnarray*}

Thus we get%
\begin{equation*}
\left \vert I_{\Omega ,\alpha }^{\left( m\right) }\left( \overrightarrow{f}%
\right) \left( x\right) \right \vert \leq C\delta ^{\epsilon }M_{\Omega
,\alpha -\epsilon }^{\left( m\right) }\left( \overrightarrow{f}\right)
\left( x\right) +C\delta ^{-\epsilon }M_{\Omega ,\alpha +\epsilon }^{\left(
m\right) }\left( \overrightarrow{f}\right) \left( x\right) .
\end{equation*}

Now we take $\delta >0$ such that%
\begin{equation*}
\delta ^{\epsilon }M_{\Omega ,\alpha -\epsilon }^{\left( m\right) }\left( 
\overrightarrow{f}\right) \left( x\right) =\delta ^{-\epsilon }M_{\Omega
,\alpha +\epsilon }^{\left( m\right) }\left( \overrightarrow{f}\right)
\left( x\right) .
\end{equation*}

This implies the Lemma \ref{lemma2*}.
\end{proof}

We now can prove the $L^{p}$ boundedness for operators $M_{\alpha }^{\left(
m\right) }$.

\begin{theorem}
\label{teo3*}Let $0<\alpha <mn$, and let $f_{i}\in L^{p_{i}}\left( {\mathbb{R%
}^{n}}\right) $ with $1\leq p_{i}\leq \infty \left( i=1,2,\ldots ,m\right) $%
, and let $\frac{1}{p}=\frac{1}{p_{1}}+\frac{1}{p_{2}}+\cdots +\frac{1}{p_{m}%
}-\frac{\alpha }{n}>0$.

$\left( i\right) $ If each $p_{i}>1$, then there exists a constant $C>0$
such that%
\begin{equation}
\left \Vert M_{\alpha }^{\left( m\right) }\left( \overrightarrow{f}\right)
\right \Vert _{L^{p}\left( {\mathbb{R}^{n}}\right) }\leq C\dprod
\limits_{i=1}^{m}\left \Vert f_{i}\right \Vert _{L^{p_{i}}\left( {\mathbb{R}%
^{n}}\right) }.  \label{3*}
\end{equation}

$\left( ii\right) $ If $p_{i}=1$ for some $i$, then there exists a constant $%
C>0$ such that%
\begin{equation}
\left \Vert M_{\alpha }^{\left( m\right) }\left( \overrightarrow{f}\right)
\right \Vert _{L^{p,\infty }\left( {\mathbb{R}^{n}}\right) }\leq C\dprod
\limits_{i=1}^{m}\left \Vert f_{i}\right \Vert _{L^{p_{i}}\left( {\mathbb{R}%
^{n}}\right) }.  \label{4*}
\end{equation}
\end{theorem}

\begin{proof}
Fix $x\in {\mathbb{R}^{n}}$, we have%
\begin{eqnarray*}
&&\frac{1}{r^{mn-\alpha }}\int \limits_{\left \vert \overrightarrow{y}\right
\vert <r}\dprod \limits_{i=1}^{m}\left \vert f_{i}\left( x-y_{i}\right)
\right \vert d\overrightarrow{y} \\
&\leq &\int \limits_{\left \vert \overrightarrow{y}\right \vert <r}\frac{1}{%
\left \vert \overrightarrow{y}\right \vert ^{mn-\alpha }}\dprod%
\limits_{i=1}^{m}\left \vert f_{i}\left( x-y_{i}\right) \right \vert d%
\overrightarrow{y} \\
&\leq &\dint \limits_{\left( {\mathbb{R}^{n}}\right) ^{m}}\frac{1}{\left
\vert \overrightarrow{y}\right \vert ^{mn-\alpha }}\dprod
\limits_{i=1}^{m}\left \vert f_{i}\left( x-y_{i}\right) \right \vert d%
\overrightarrow{y} \\
&=&I_{\alpha }^{\left( m\right) }\left( \left \vert f_{1}\right \vert ,\left
\vert f_{2}\right \vert ,\ldots ,\left \vert f_{m}\right \vert \right)
\left( x\right) .
\end{eqnarray*}

Taking the supremum for $r>0$ on both sides of the above inequality, we have%
\begin{equation*}
M_{\alpha }^{\left( m\right) }\left( \overrightarrow{f}\right) \left(
x\right) \leq I_{\alpha }^{\left( m\right) }\left( \left \vert f_{1}\right
\vert ,\left \vert f_{2}\right \vert ,\ldots ,\left \vert f_{m}\right \vert
\right) \left( x\right) .
\end{equation*}

Applying Theorem 1 in \cite{Kenig} and from the above inequality we
immediately obtain the inequalities (\ref{3*}) and (\ref{4*}).
\end{proof}

\textbf{The proof of Theorem \ref{teo1}. }

\begin{proof}
If each $s^{\prime }<p_{i}$, we can apply Lemma \ref{lemma1*} and Theorem %
\ref{teo3*} to get%
\begin{eqnarray*}
\left \Vert M_{\Omega ,\alpha }^{\left( m\right) }\left( \overrightarrow{f}%
\right) \right \Vert _{L^{p}\left( {\mathbb{R}^{n}}\right) } &=&\left( \dint
\limits_{{\mathbb{R}^{n}}}\left \vert M_{\Omega ,\alpha }^{\left( m\right)
}\left( \overrightarrow{f}\right) \left( x\right) \right \vert ^{p}dx\right)
^{\frac{1}{p}} \\
&\leq &C\left( \dint \limits_{{\mathbb{R}^{n}}}\left \vert M_{\alpha
s^{\prime }}^{\left( m\right) }\left( \left \vert f_{1}\right \vert
^{s^{\prime }},\left \vert f_{2}\right \vert ^{s^{\prime }},\ldots ,\left
\vert f_{m}\right \vert ^{s^{\prime }}\right) \left( x\right) \right \vert ^{%
\frac{p}{s^{\prime }}}dx\right) ^{\frac{1}{p}} \\
&\leq &C\dprod \limits_{i=1}^{m}\left \Vert \left \vert f_{i}\right \vert
^{s^{\prime }}\right \Vert _{L^{\frac{p_{i}}{s^{\prime }}}\left( {\mathbb{R}%
^{n}}\right) }^{\frac{1}{s^{\prime }}}\leq C\dprod \limits_{i=1}^{m}\left
\Vert f_{i}\right \Vert _{L^{p_{i}}\left( {\mathbb{R}^{n}}\right) }.
\end{eqnarray*}%
If $p_{i}=s^{\prime }$ for some $i$, we also applying Lemma \ref{lemma1*}
and Theorem \ref{teo3*} to get for any $\lambda >0$ that%
\begin{eqnarray*}
&&\left \vert \left \{ x:\left \vert M_{\Omega ,\alpha }^{\left( m\right)
}\left( \overrightarrow{f}\right) \left( x\right) \right \vert >\lambda
\right \} \right \vert \\
&\leq &C\left \vert \left \{ x:\left \vert \left[ M_{\alpha s^{\prime
}}^{\left( m\right) }\left( \left \vert f_{1}\right \vert ^{s^{\prime
}},\left \vert f_{2}\right \vert ^{s^{\prime }},\ldots ,\left \vert
f_{m}\right \vert ^{s^{\prime }}\right) \left( x\right) \right] \right \vert
^{\frac{1}{s^{\prime }}}>\lambda \right \} \right \vert \\
&\leq &C\left \vert \left \{ x:\left \vert \left[ M_{\alpha s^{\prime
}}^{\left( m\right) }\left( \left \vert f_{1}\right \vert ^{s^{\prime
}},\left \vert f_{2}\right \vert ^{s^{\prime }},\ldots ,\left \vert
f_{m}\right \vert ^{s^{\prime }}\right) \left( x\right) \right] \right \vert
>\lambda ^{s^{\prime }}\right \} \right \vert \\
&\leq &C\left( \frac{1}{\lambda ^{s^{\prime }}}\dprod \limits_{i=1}^{m}\left
\Vert \left \vert f_{i}\right \vert ^{s^{\prime }}\right \Vert _{L^{\frac{%
p_{i}}{s^{\prime }}}\left( {\mathbb{R}^{n}}\right) }\right) ^{\frac{p}{%
s^{\prime }}}\leq C\left( \frac{1}{\lambda }\dprod \limits_{i=1}^{m}\left
\Vert f_{i}\right \Vert _{L^{p_{i}}\left( {\mathbb{R}^{n}}\right) }\right)
^{p}.
\end{eqnarray*}%
Thus, we complete the proof of Theorem \ref{teo1}.
\end{proof}

\textbf{The proof of Theorem \ref{teo2}. }

\begin{proof}
Take a small positive number $\epsilon $ with $0<\epsilon <\min \left \{
\alpha ,\frac{mn}{s^{\prime }}-\alpha ,\frac{n}{p}\right \} $. One can then
see from the conditions of Theorem \ref{teo2} that $1\leq s^{\prime }<\frac{%
mn}{\alpha +\epsilon }$ and $1\leq s^{\prime }<\frac{mn}{\alpha -\epsilon }$%
, and that%
\begin{equation}
\frac{1}{q_{1}}:=\frac{1}{p_{1}}+\frac{1}{p_{2}}+\cdots +\frac{1}{p_{m}}-%
\frac{\alpha +\epsilon }{n}=\frac{1}{p}-\frac{\epsilon }{n}>0,  \label{5*}
\end{equation}%
\begin{equation}
\frac{1}{q_{2}}:=\frac{1}{p_{1}}+\frac{1}{p_{2}}+\cdots +\frac{1}{p_{m}}-%
\frac{\alpha -\epsilon }{n}=\frac{1}{p}+\frac{\epsilon }{n}>0.  \label{6*}
\end{equation}%
Now if each $s^{\prime }<p_{i}$, then Theorem \ref{teo1} implies that%
\begin{equation*}
\left \Vert M_{\Omega ,\alpha +\epsilon }^{\left( m\right) }\left( 
\overrightarrow{f}\right) \right \Vert _{L^{q_{1}}\left( {\mathbb{R}^{n}}%
\right) }\leq C\dprod \limits_{i=1}^{m}\left \Vert f_{i}\right \Vert
_{L^{p_{i}}\left( {\mathbb{R}^{n}}\right) },
\end{equation*}%
\begin{equation*}
\left \Vert M_{\Omega ,\alpha -\epsilon }^{\left( m\right) }\left( 
\overrightarrow{f}\right) \right \Vert _{L^{q_{2}}\left( {\mathbb{R}^{n}}%
\right) }\leq C\dprod \limits_{i=1}^{m}\left \Vert f_{i}\right \Vert
_{L^{p_{i}}\left( {\mathbb{R}^{n}}\right) }.
\end{equation*}%
Note that $\frac{p}{2q_{1}}+\frac{p}{2q_{2}}=1$. Using Lemma \ref{lemma2*},
the H\"{o}lder's inequality and the above two inequalities, we obtain%
\begin{eqnarray*}
\left \Vert I_{\Omega ,\alpha }^{\left( m\right) }\left( \overrightarrow{f}%
\right) \right \Vert _{L^{p}\left( {\mathbb{R}^{n}}\right) } &=&\left( \dint
\limits_{{\mathbb{R}^{n}}}\left \vert I_{\Omega ,\alpha }^{\left( m\right)
}\left( \overrightarrow{f}\right) \left( x\right) \right \vert ^{p}dx\right)
^{\frac{1}{p}} \\
&\leq &C\left( \dint \limits_{{\mathbb{R}^{n}}}\left[ M_{\Omega ,\alpha
+\epsilon }^{\left( m\right) }\left( \overrightarrow{f}\right) \left(
x\right) \right] ^{\frac{p}{2}}\left[ M_{\Omega ,\alpha -\epsilon }^{\left(
m\right) }\left( \overrightarrow{f}\right) \left( x\right) \right] ^{\frac{p%
}{2}}dx\right) ^{\frac{1}{p}} \\
&=&C\left( \dint \limits_{{\mathbb{R}^{n}}}\left[ M_{\Omega ,\alpha
+\epsilon }^{\left( m\right) }\left( \overrightarrow{f}\right) \left(
x\right) \right] ^{q_{1}}dx\right) ^{\frac{1}{2q_{1}}}\left( \dint \limits_{{%
\mathbb{R}^{n}}}\left[ M_{\Omega ,\alpha -\epsilon }^{\left( m\right)
}\left( \overrightarrow{f}\right) \left( x\right) \right] ^{q_{2}}dx\right)
^{\frac{1}{2q_{2}}} \\
&=&\left \Vert M_{\Omega ,\alpha +\epsilon }^{\left( m\right) }\left( 
\overrightarrow{f}\right) \right \Vert _{L^{q_{1}}\left( {\mathbb{R}^{n}}%
\right) }^{\frac{1}{2}}\left \Vert M_{\Omega ,\alpha -\epsilon }^{\left(
m\right) }\left( \overrightarrow{f}\right) \right \Vert _{L^{q_{2}}\left( {%
\mathbb{R}^{n}}\right) }^{\frac{1}{2}} \\
&\leq &C\dprod \limits_{i=1}^{m}\left \Vert f_{i}\right \Vert
_{L^{p_{i}}\left( {\mathbb{R}^{n}}\right) }.
\end{eqnarray*}%
This is the desired inequality (\ref{3}) of Theorem \ref{teo2}.

Similarly, If $p_{i}=s^{\prime }$ for some $i$, then Theorem \ref{teo1}
implies that%
\begin{equation*}
\left \Vert M_{\Omega ,\alpha +\epsilon }^{\left( m\right) }\left( 
\overrightarrow{f}\right) \right \Vert _{L^{q_{1},\infty }\left( {\mathbb{R}%
^{n}}\right) }\leq C\dprod \limits_{i=1}^{m}\left \Vert f_{i}\right \Vert
_{L^{p_{i}}\left( {\mathbb{R}^{n}}\right) },
\end{equation*}%
\begin{equation*}
\left \Vert M_{\Omega ,\alpha -\epsilon }^{\left( m\right) }\left( 
\overrightarrow{f}\right) \right \Vert _{L^{q_{2},\infty }\left( {\mathbb{R}%
^{n}}\right) }\leq C\dprod \limits_{i=1}^{m}\left \Vert f_{i}\right \Vert
_{L^{p_{i}}\left( {\mathbb{R}^{n}}\right) }.
\end{equation*}%
For any $\lambda >0$, we take $w^{2}=\lambda ^{\frac{2q_{1}}{q_{1}+q_{2}}%
}\left( \dprod \limits_{i=1}^{m}\left \Vert f_{i}\right \Vert
_{L^{p_{i}}}\right) ^{\frac{q_{2}-q_{1}}{q_{2}+q_{1}}}$, then Lemma \ref%
{lemma2*} and the above two inequalities give that%
\begin{eqnarray*}
&&\left \vert \left \{ x:\left \vert I_{\Omega ,\alpha }^{\left( m\right)
}\left( \overrightarrow{f}\right) \left( x\right) \right \vert >\lambda
\right \} \right \vert \\
&\leq &\left \vert \left \{ x:\left \vert M_{\Omega ,\alpha -\epsilon
}^{\left( m\right) }\left( \overrightarrow{f}\right) \left( x\right) \right
\vert ^{\frac{1}{2}}>\frac{w}{C}\right \} \right \vert +\left \vert \left \{
x:\left \vert M_{\Omega ,\alpha +\epsilon }^{\left( m\right) }\left( 
\overrightarrow{f}\right) \left( x\right) \right \vert ^{\frac{1}{2}}>\frac{%
\lambda }{w}\right \} \right \vert \\
&\leq &\left \vert \left \{ x:\left \vert M_{\Omega ,\alpha -\epsilon
}^{\left( m\right) }\left( \overrightarrow{f}\right) \left( x\right) \right
\vert ^{\frac{1}{2}}>\frac{w^{2}}{C^{2}}\right \} \right \vert +\left \vert
\left \{ x:\left \vert M_{\Omega ,\alpha +\epsilon }^{\left( m\right)
}\left( \overrightarrow{f}\right) \left( x\right) \right \vert ^{\frac{1}{2}%
}>\frac{\lambda ^{2}}{w^{2}}\right \} \right \vert \\
&\leq &C\left( \frac{1}{w^{2}}\dprod \limits_{i=1}^{m}\left \Vert
f_{i}\right \Vert _{L^{p_{i}}\left( {\mathbb{R}^{n}}\right) }\right)
^{q_{2}}+C\left( \frac{w^{2}}{\lambda ^{2}}\dprod \limits_{i=1}^{m}\left
\Vert f_{i}\right \Vert _{L^{p_{i}}\left( {\mathbb{R}^{n}}\right) }\right)
^{q_{1}} \\
&\leq &C\left( \frac{1}{\lambda }\dprod \limits_{i=1}^{m}\left \Vert
f_{i}\right \Vert _{L^{p_{i}}\left( {\mathbb{R}^{n}}\right) }\right) ^{p}.
\end{eqnarray*}%
This yields the desired inequality (\ref{4}) of Theorem \ref{teo2}. Thus we
complete the proof of Theorem \ref{teo2}.
\end{proof}

\section{Product weighted $L^{p}$-estimates for $I_{\Omega ,\protect \alpha %
}^{\left( m\right) }$ and $M_{\Omega ,\protect \alpha }^{\left( m\right) }$}

For the proofs of Theorems \ref{teo3} and \ref{teo4}, we need some lemmas.

\begin{lemma}
\label{lemma31}\cite{Ding} Suppose that $0<\alpha <n$, $1\leq s^{\prime }<p<%
\frac{n}{\alpha }$, $\frac{1}{q}=\frac{1}{p}-\frac{\alpha }{n}$ and $w\left(
x\right) ^{s^{\prime }}\in A_{\frac{p}{s^{\prime }},\frac{q}{s^{\prime }}}$.
Then there exists a small positive number $\epsilon $ with $0<\epsilon <\min
\left \{ \alpha ,\frac{n}{p}-\alpha ,\frac{n}{q^{\prime }}\right \} $ such
that $w\left( x\right) ^{s^{\prime }}\in A_{\frac{p}{s^{\prime }},\frac{%
q_{\epsilon }}{s^{\prime }}}$ and $w\left( x\right) ^{s^{\prime }}\in A_{%
\frac{p}{s^{\prime }},\frac{\tilde{q}_{\epsilon }}{s^{\prime }}}$, where $%
\frac{1}{\tilde{q}_{\epsilon }}=\frac{1}{p}-\frac{\alpha +\epsilon }{n}$ and 
$\frac{1}{\tilde{q}_{\epsilon }}=\frac{1}{p}-\frac{\alpha -\epsilon }{n}$.
\end{lemma}

\begin{lemma}
\label{lemma34}\cite{Shi} Let $0<\alpha <mn$, $1<p_{i}<\frac{mn}{\alpha }$, $%
\frac{1}{p}=\frac{1}{p_{1}}+\frac{1}{p_{2}}+\cdots +\frac{1}{p_{m}}-\frac{%
\alpha }{n}>0$, $\frac{1}{q_{i}}=\frac{1}{p_{i}}-\frac{\alpha }{mn}$, $%
i=1,\ldots ,m$. Assume that $f_{i}\in L_{w^{p_{i}}}^{p_{i}}\left( {\mathbb{R}%
^{n}}\right) $ with weight $w\left( x\right) \in \dbigcap
\limits_{i=1}^{m}A_{p_{i},q_{i}}$, then 
\begin{equation*}
\left \Vert M_{\alpha }^{\left( m\right) }\left( \overrightarrow{f}\right)
\right \Vert _{L_{w^{p}}^{p}\left( {\mathbb{R}^{n}}\right) }\leq C\dprod
\limits_{i=1}^{m}\left \Vert f_{i}\right \Vert _{L_{w^{p_{i}}}^{p_{i}}\left( 
{\mathbb{R}^{n}}\right) },
\end{equation*}%
with the absolute constant $C$.
\end{lemma}

\textbf{The proof of Theorem \ref{teo3}. }

\begin{proof}
It's easy to see that $\frac{1}{\frac{p}{s^{\prime }}}=\frac{1}{\frac{p_{1}}{%
s^{\prime }}}+\frac{1}{\frac{p_{2}}{s^{\prime }}}+\cdots +\frac{1}{\frac{%
p_{m}}{s^{\prime }}}-\frac{\alpha s^{\prime }}{n}>0$, $\frac{1}{\frac{q_{i}}{%
s^{\prime }}}=\frac{1}{\frac{p_{i}}{s^{\prime }}}-\frac{\alpha s^{\prime }}{%
mn}$ and $1<\frac{p_{i}}{s^{\prime }}<\frac{mn}{\alpha s^{\prime }}$.
Therefore, by Lemma \ref{lemma1*} and Lemma \ref{lemma34} we get%
\begin{eqnarray*}
\left \Vert M_{\Omega ,\alpha }^{\left( m\right) }\left( \overrightarrow{f}%
\right) \right \Vert _{L_{w^{p}}^{p}\left( {\mathbb{R}^{n}}\right) }
&=&\left( \dint \limits_{{\mathbb{R}^{n}}}\left \vert M_{\Omega ,\alpha
}^{\left( m\right) }\left( \overrightarrow{f}\right) \left( x\right) w\left(
x\right) \right \vert ^{p}dx\right) ^{\frac{1}{p}} \\
&\leq &C\left( \dint \limits_{{\mathbb{R}^{n}}}\left \vert M_{\alpha
s^{\prime }}^{\left( m\right) }\left( \left \vert f_{1}\right \vert
^{s^{\prime }},\left \vert f_{2}\right \vert ^{s^{\prime }},\ldots ,\left
\vert f_{m}\right \vert ^{s^{\prime }}\right) \left( x\right) w\left(
x\right) ^{s^{\prime }}\right \vert ^{\frac{p}{s^{\prime }}}dx\right) ^{%
\frac{1}{p}} \\
&\leq &C\dprod \limits_{i=1}^{m}\left \Vert \left \vert f_{i}\right \vert
^{s^{\prime }}\right \Vert _{L_{w^{p_{i}}}^{\frac{p_{i}}{s^{\prime }}}\left( 
{\mathbb{R}^{n}}\right) }^{\frac{1}{s^{\prime }}}\leq C\dprod
\limits_{i=1}^{m}\left \Vert f_{i}\right \Vert _{L_{w^{p_{i}}}^{p_{i}}\left( 
{\mathbb{R}^{n}}\right) }.
\end{eqnarray*}%
This proves Theorem \ref{teo3}.
\end{proof}

\textbf{The proof of Theorem \ref{teo4}. }

\begin{proof}
Since $0<\alpha <mn$, $s^{\prime }<p_{i}<\frac{mn}{\alpha }$, $\frac{1}{q_{i}%
}=\frac{1}{p_{i}}-\frac{\alpha }{mn}$and $w\left( x\right) ^{s^{\prime }}\in
\dbigcap \limits_{i=1}^{n}A_{\frac{p}{s^{\prime }},\frac{q}{s^{\prime }}}$,
we get by Lemma \ref{lemma31} that there exists a small positive number $%
\epsilon $ such that $w\left( x\right) ^{s^{\prime }}\in \dbigcap
\limits_{i=1}^{n}A_{\frac{p_{i}}{s^{\prime }},\frac{\gamma _{i}}{s^{\prime }}%
}$ and $w\left( x\right) ^{s^{\prime }}\in \dbigcap \limits_{i=1}^{n}A_{%
\frac{p_{i}}{s^{\prime }},\frac{\xi _{i}}{s^{\prime }}}$, where $\frac{1}{%
\gamma _{i}}=\frac{1}{p_{i}}-\frac{\alpha +\epsilon }{mn}$, $\frac{1}{\xi
_{i}}=\frac{1}{p_{i}}-\frac{\alpha -\epsilon }{mn}$. Put%
\begin{equation*}
\frac{1}{\beta _{1}}:=\frac{1}{p_{1}}+\frac{1}{p_{2}}+\cdots +\frac{1}{p_{m}}%
-\frac{\alpha +\epsilon }{n}=\frac{1}{p}-\frac{\epsilon }{n}>0,
\end{equation*}%
\begin{equation*}
\frac{1}{\beta _{2}}:=\frac{1}{p_{1}}+\frac{1}{p_{2}}+\cdots +\frac{1}{p_{m}}%
-\frac{\alpha -\epsilon }{n}=\frac{1}{p}+\frac{\epsilon }{n}>0.
\end{equation*}%
Then, by Theorem \ref{teo3}, we get%
\begin{equation*}
\left \Vert M_{\Omega ,\alpha +\epsilon }^{\left( m\right) }\left( 
\overrightarrow{f}\right) \right \Vert _{L_{w^{\beta _{1}}}^{\beta
_{1}}\left( {\mathbb{R}^{n}}\right) }\leq C\dprod \limits_{i=1}^{m}\left
\Vert f_{i}\right \Vert _{L_{w^{p_{i}}}^{p_{i}}\left( {\mathbb{R}^{n}}%
\right) },
\end{equation*}%
\begin{equation*}
\left \Vert M_{\Omega ,\alpha -\epsilon }^{\left( m\right) }\left( 
\overrightarrow{f}\right) \right \Vert _{L_{w^{\beta _{2}}}^{\beta
_{2}}\left( {\mathbb{R}^{n}}\right) }\leq C\dprod \limits_{i=1}^{m}\left
\Vert f_{i}\right \Vert _{L_{w^{p_{i}}}^{p_{i}}\left( {\mathbb{R}^{n}}%
\right) }.
\end{equation*}%
Now by Lemma \ref{lemma2*}, the H\"{o}lder's inequality and the above two
inequalities we get%
\begin{eqnarray*}
\left \Vert I_{\Omega ,\alpha }^{\left( m\right) }\left( \overrightarrow{f}%
\right) \right \Vert _{L_{w^{p}}^{p}\left( {\mathbb{R}^{n}}\right) }
&=&\left( \dint \limits_{{\mathbb{R}^{n}}}\left \vert I_{\Omega ,\alpha
}^{\left( m\right) }\left( \overrightarrow{f}\right) \left( x\right) w\left(
x\right) \right \vert ^{p}dx\right) ^{\frac{1}{p}} \\
&\leq &C\left( \dint \limits_{{\mathbb{R}^{n}}}\left[ M_{\Omega ,\alpha
+\epsilon }^{\left( m\right) }\left( \overrightarrow{f}\right) \left(
x\right) w\left( x\right) \right] ^{\frac{p}{2}}\left[ M_{\Omega ,\alpha
-\epsilon }^{\left( m\right) }\left( \overrightarrow{f}\right) \left(
x\right) w\left( x\right) \right] ^{\frac{p}{2}}dx\right) ^{\frac{1}{p}} \\
&\leq &\left \Vert M_{\Omega ,\alpha +\epsilon }^{\left( m\right) }\left( 
\overrightarrow{f}\right) \right \Vert _{L_{w^{\beta _{1}}}^{\beta
_{1}}\left( {\mathbb{R}^{n}}\right) }^{\frac{1}{2}}\left \Vert M_{\Omega
,\alpha -\epsilon }^{\left( m\right) }\left( \overrightarrow{f}\right)
\right \Vert _{L_{w^{\beta _{2}}}^{\beta _{2}}\left( {\mathbb{R}^{n}}\right)
}^{\frac{1}{2}} \\
&\leq &C\dprod \limits_{i=1}^{m}\left \Vert f_{i}\right \Vert
_{L_{w^{p_{i}}}^{p_{i}}\left( {\mathbb{R}^{n}}\right) }.
\end{eqnarray*}%
This completes the proof of Theorem \ref{teo4}.
\end{proof}

\section{Two-weighted estimates for $I_{\Omega ,\protect \alpha }^{\left(
m\right) }$ and $M_{\Omega ,\protect \alpha }^{\left( m\right) }$}

For the proofs of Theorems \ref{teo5} and \ref{teo6}, we need following
lemma.

\begin{lemma}
\label{lemma42}\cite{Shi} Let $0<\alpha <mn$, $1<p_{i}<q_{i}<\infty $ for
each $i=1,2,\ldots ,m$. Let $\frac{1}{p}=\frac{1}{q_{1}}+\frac{1}{q_{2}}%
+\cdots +\frac{1}{q_{m}}$ and $\left( u,\upsilon \right) $ be a pair of
weights. If there exists $r_{i}>1$ such that, for every cube $Q$ in ${%
\mathbb{R}^{n}}$,%
\begin{equation*}
\left \vert Q\right \vert ^{\frac{1}{q_{i}}+\frac{\alpha }{mn}-\frac{1}{p_{i}%
}}\left( \frac{1}{\left \vert Q\right \vert }\dint \limits_{Q}u\left(
x\right) dx\right) ^{\frac{1}{q_{i}}}\left( \frac{1}{\left \vert Q\right
\vert }\dint \limits_{Q}\upsilon \left( x\right) ^{r_{i}\left(
1-p_{i}^{\prime }\right) }dx\right) ^{\frac{1}{r_{i}p_{i}^{\prime }}}\leq C,
\end{equation*}%
then for arbitrary $f_{i}\in L_{\upsilon }^{p_{i}}\left( {\mathbb{R}^{n}}%
\right) $, there is a constant $C>0$, independent of $f_{i}$, such that 
\begin{equation}
\left \Vert M_{\alpha }^{\left( m\right) }\left( \overrightarrow{f}\right)
\right \Vert _{L_{u}^{p}\left( {\mathbb{R}^{n}}\right) }\leq C\dprod
\limits_{i=1}^{m}\left \Vert f_{i}\right \Vert _{L_{\upsilon }^{p_{i}}\left( 
{\mathbb{R}^{n}}\right) }.  \label{41}
\end{equation}
\end{lemma}

\textbf{The proof of Theorem \ref{teo5}. }

\begin{proof}
From the conditions of Theorem \ref{teo5}, one notes that $\left \vert
f_{i}\right \vert ^{s^{\prime }}\in L_{\upsilon }^{\frac{p_{i}}{s^{\prime }}%
}\left( {\mathbb{R}^{n}}\right) $, $\frac{1}{\frac{p}{s^{\prime }}}=\frac{1}{%
\frac{q_{1}}{s^{\prime }}}+\frac{1}{\frac{q_{2}}{s^{\prime }}}+\cdots +\frac{%
1}{\frac{q_{m}}{s^{\prime }}}$. Since $s^{\prime }<p_{i}<q_{i}<\infty $ for
each $i$, then $\frac{p_{i}}{s^{\prime }}>1$. Lemma \ref{lemma1*} and Lemma %
\ref{lemma42} imply that%
\begin{eqnarray*}
\left \Vert M_{\Omega ,\alpha }^{\left( m\right) }\left( \overrightarrow{f}%
\right) \right \Vert _{L_{u}^{p}\left( {\mathbb{R}^{n}}\right) } &=&\left(
\dint \limits_{{\mathbb{R}^{n}}}\left \vert M_{\Omega ,\alpha }^{\left(
m\right) }\left( \overrightarrow{f}\right) \left( x\right) \right \vert
^{p}u\left( x\right) dx\right) ^{\frac{1}{p}} \\
&\leq &C\left( \dint \limits_{{\mathbb{R}^{n}}}\left \vert M_{\alpha
s^{\prime }}^{\left( m\right) }\left( \left \vert f_{1}\right \vert
^{s^{\prime }},\left \vert f_{2}\right \vert ^{s^{\prime }},\ldots ,\left
\vert f_{m}\right \vert ^{s^{\prime }}\right) \left( x\right) \right \vert ^{%
\frac{p}{s^{\prime }}}u\left( x\right) dx\right) ^{\frac{1}{p}} \\
&=&\left \Vert M_{\alpha s^{\prime }}^{\left( m\right) }\left( \left \vert
f_{1}\right \vert ^{s^{\prime }},\left \vert f_{2}\right \vert ^{s^{\prime
}},\ldots ,\left \vert f_{m}\right \vert ^{s^{\prime }}\right) \right \Vert
_{L_{u}^{\frac{p}{s^{\prime }}}\left( {\mathbb{R}^{n}}\right) }^{\frac{1}{%
s^{\prime }}} \\
&\leq &C\dprod \limits_{i=1}^{m}\left \Vert \left \vert f_{i}\right \vert
^{s^{\prime }}\right \Vert _{L_{\upsilon }^{\frac{p_{i}}{s^{\prime }}}\left( 
{\mathbb{R}^{n}}\right) }^{\frac{1}{s^{\prime }}}\leq C\dprod
\limits_{i=1}^{m}\left \Vert f_{i}\right \Vert _{L_{\upsilon }^{p_{i}}\left( 
{\mathbb{R}^{n}}\right) }.
\end{eqnarray*}%
This proves inequality (\ref{8}).
\end{proof}

\textbf{The proof of Theorem \ref{teo6}. }

\begin{proof}
Under the conditions of Theorem \ref{teo6}, we can choose a positive number $%
\epsilon $ such that%
\begin{equation*}
0<\epsilon <\min \left \{ \alpha ,\frac{mn}{s^{\prime }}-\alpha ,\frac{n}{p}%
,mn\left( \frac{1}{p_{i}}-\frac{1}{mp}\right) ,\frac{1}{pr_{i}^{\prime }}%
\right \} .
\end{equation*}%
Let $\frac{1}{l_{1}}=\frac{1}{p}-\frac{\epsilon }{n}$, $\frac{1}{l_{2}}=%
\frac{1}{p}+\frac{\epsilon }{n}$, then $\frac{p}{2l_{1}}+\frac{p}{2l_{2}}=1$%
. Applying Lemma \ref{lemma2*} and the H\"{o}lder's inequality, we get%
\begin{eqnarray*}
\left \Vert I_{\Omega ,\alpha }^{\left( m\right) }\left( \overrightarrow{f}%
\right) \right \Vert _{L_{u}^{p}\left( {\mathbb{R}^{n}}\right) } &=&\left(
\dint \limits_{{\mathbb{R}^{n}}}\left \vert I_{\Omega ,\alpha }^{\left(
m\right) }\left( \overrightarrow{f}\right) \left( x\right) \right \vert
^{p}u\left( x\right) dx\right) ^{\frac{1}{p}} \\
&\leq &C\left( \dint \limits_{{\mathbb{R}^{n}}}\left[ M_{\Omega ,\alpha
+\epsilon }^{\left( m\right) }\left( \overrightarrow{f}\right) \left(
x\right) \right] ^{\frac{p}{2}}\left[ M_{\Omega ,\alpha -\epsilon }^{\left(
m\right) }\left( \overrightarrow{f}\right) \left( x\right) \right] ^{\frac{p%
}{2}}u\left( x\right) dx\right) ^{\frac{1}{p}} \\
&\leq &C\left( \dint \limits_{{\mathbb{R}^{n}}}\left[ M_{\Omega ,\alpha
+\epsilon }^{\left( m\right) }\left( \overrightarrow{f}\right) \left(
x\right) \right] ^{l_{1}}u\left( x\right) ^{\frac{l_{1}}{p}}dx\right) ^{%
\frac{1}{2l_{1}}}\times \\
&&\left( \dint \limits_{{\mathbb{R}^{n}}}\left[ M_{\Omega ,\alpha -\epsilon
}^{\left( m\right) }\left( \overrightarrow{f}\right) \left( x\right) \right]
^{l_{2}}u\left( x\right) ^{\frac{l_{2}}{p}}dx\right) ^{\frac{1}{2l_{2}}} \\
&\leq &C\left \Vert M_{\Omega ,\alpha +\epsilon }^{\left( m\right) }\left( 
\overrightarrow{f}\right) \right \Vert _{L_{u^{\frac{l_{1}}{p}%
}}^{l_{1}}\left( {\mathbb{R}^{n}}\right) }^{\frac{1}{2}}\left \Vert
M_{\Omega ,\alpha -\epsilon }^{\left( m\right) }\left( \overrightarrow{f}%
\right) \right \Vert _{L_{u^{\frac{l_{2}}{p}}}^{l_{2}}\left( {\mathbb{R}^{n}}%
\right) }^{\frac{1}{2}}.
\end{eqnarray*}%
Suppose that $\frac{1}{g_{i}}=\frac{1}{mp}-\frac{\epsilon }{mn}$, $\frac{1}{%
h_{i}}=\frac{1}{mp}+\frac{\epsilon }{mn}$. Then due to the way we choose $%
\epsilon $, we have $s^{\prime }<p_{i}<g_{i}<\infty $, $s^{\prime
}<p_{i}<h_{i}<\infty $, and moreover%
\begin{equation*}
\frac{1}{l_{1}}=\frac{1}{g_{1}}+\frac{1}{g_{2}}+\cdots +\frac{1}{g_{m}}\text{%
, }\frac{1}{l_{2}}=\frac{1}{h_{1}}+\frac{1}{h_{2}}+\cdots +\frac{1}{h_{m}}.
\end{equation*}%
Observe that $1<\frac{l_{1}}{p}<r_{i}$, we have%
\begin{eqnarray*}
&&\left \vert Q\right \vert ^{\frac{s^{\prime }}{g_{i}}+\frac{\left( \alpha
+\epsilon \right) s^{\prime }}{mn}-\frac{s^{\prime }}{p_{i}}}\left( \frac{1}{%
\left \vert Q\right \vert }\dint \limits_{Q}u\left( x\right) ^{\frac{l_{1}}{p%
}}dx\right) ^{\frac{s^{\prime }}{g_{i}}}\left( \frac{1}{\left \vert Q\right
\vert }\dint \limits_{Q}\upsilon \left( x\right) ^{r_{i}\left( 1-\left( 
\frac{p_{i}}{s^{\prime }}\right) ^{\prime }\right) }dx\right) ^{\frac{1}{%
r_{i}\left( \frac{p_{i}}{s^{\prime }}\right) ^{\prime }}} \\
&\leq &\left \vert Q\right \vert ^{\frac{s^{\prime }}{mp}+\frac{\alpha
s^{\prime }}{mn}-\frac{s^{\prime }}{p_{i}}}\left( \frac{1}{\left \vert
Q\right \vert }\dint \limits_{Q}u\left( x\right) ^{r_{i}}dx\right) ^{\frac{%
s^{\prime }}{r_{i}mp}}\left( \frac{1}{\left \vert Q\right \vert }\dint
\limits_{Q}\upsilon \left( x\right) ^{r_{i}\left( 1-\left( \frac{p_{i}}{%
s^{\prime }}\right) ^{\prime }\right) }dx\right) ^{\frac{1}{r_{i}\left( 
\frac{p_{i}}{s^{\prime }}\right) ^{\prime }}} \\
&\leq &C.
\end{eqnarray*}%
Thus the pair of weights $\left( u^{\frac{l_{1}}{p}},\upsilon \right) $
satisfies condition (\ref{7}) with $s^{\prime }<p_{i}<g_{i}<\infty $ and $%
\alpha +\epsilon $. Then, Theorem \ref{teo5} implies that $M_{\Omega ,\alpha
+\epsilon }^{\left( m\right) }$ is bounded from $L_{\upsilon }^{p_{1}}\left( 
{\mathbb{R}^{n}}\right) \times L_{\upsilon }^{p_{2}}\left( {\mathbb{R}^{n}}%
\right) \times \cdots \times L_{\upsilon }^{p_{m}}\left( {\mathbb{R}^{n}}%
\right) $ to $L_{u^{\frac{l_{1}}{p}}}^{l_{1}}\left( {\mathbb{R}^{n}}\right) $%
. On the other hand, we can see $\frac{l_{2}}{p}<1<r_{i}$ and so%
\begin{eqnarray*}
&&\left \vert Q\right \vert ^{\frac{s^{\prime }}{h_{i}}+\frac{\left( \alpha
-\epsilon \right) s^{\prime }}{mn}-\frac{s^{\prime }}{p_{i}}}\left( \frac{1}{%
\left \vert Q\right \vert }\dint \limits_{Q}u\left( x\right) ^{\frac{l_{2}}{p%
}}dx\right) ^{\frac{s^{\prime }}{h_{i}}}\left( \frac{1}{\left \vert Q\right
\vert }\dint \limits_{Q}\upsilon \left( x\right) ^{r_{i}\left( 1-\left( 
\frac{p_{i}}{s^{\prime }}\right) ^{\prime }\right) }dx\right) ^{\frac{1}{%
r_{i}\left( \frac{p_{i}}{s^{\prime }}\right) ^{\prime }}} \\
&\leq &\left \vert Q\right \vert ^{\frac{s^{\prime }}{mp}+\frac{\alpha
s^{\prime }}{mn}-\frac{s^{\prime }}{p_{i}}}\left( \frac{1}{\left \vert
Q\right \vert }\dint \limits_{Q}u\left( x\right) ^{r_{i}}dx\right) ^{\frac{%
s^{\prime }}{r_{i}mp}}\left( \frac{1}{\left \vert Q\right \vert }\dint
\limits_{Q}\upsilon \left( x\right) ^{r_{i}\left( 1-\left( \frac{p_{i}}{%
s^{\prime }}\right) ^{\prime }\right) }dx\right) ^{\frac{1}{r_{i}\left( 
\frac{p_{i}}{s^{\prime }}\right) ^{\prime }}} \\
&\leq &C.
\end{eqnarray*}%
In this case, the pair of weights $\left( u^{\frac{l_{2}}{p}},\upsilon
\right) $ verifies condition (\ref{7}) with $s^{\prime }<p_{i}<h_{i}<\infty $
and $\alpha -\epsilon $. Then, Theorem \ref{teo5} implies that $M_{\Omega
,\alpha -\epsilon }^{\left( m\right) }$ is a bounded operator from $%
L_{\upsilon }^{p_{1}}\left( {\mathbb{R}^{n}}\right) \times L_{\upsilon
}^{p_{2}}\left( {\mathbb{R}^{n}}\right) \times \cdots \times L_{\upsilon
}^{p_{m}}\left( {\mathbb{R}^{n}}\right) $ to $L_{u^{\frac{l_{2}}{p}%
}}^{l_{1}}\left( {\mathbb{R}^{n}}\right) $.Combining above estimates
together, we get%
\begin{eqnarray*}
\left \Vert I_{\Omega ,\alpha }^{\left( m\right) }\left( \overrightarrow{f}%
\right) \right \Vert _{L_{u}^{p}\left( {\mathbb{R}^{n}}\right) } &\leq
&C\left \Vert M_{\Omega ,\alpha +\epsilon }^{\left( m\right) }\left( 
\overrightarrow{f}\right) \right \Vert _{L_{u^{\frac{l_{1}}{p}%
}}^{l_{1}}\left( {\mathbb{R}^{n}}\right) }^{\frac{1}{2}}\left \Vert
M_{\Omega ,\alpha -\epsilon }^{\left( m\right) }\left( \overrightarrow{f}%
\right) \right \Vert _{L_{u^{\frac{l_{2}}{p}}}^{l_{2}}\left( {\mathbb{R}^{n}}%
\right) }^{\frac{1}{2}} \\
&\leq &\dprod \limits_{i=1}^{m}\left \Vert f_{i}\right \Vert _{L_{\upsilon
}^{p_{i}}\left( {\mathbb{R}^{n}}\right) }.
\end{eqnarray*}%
This is desired inequality (\ref{9}) of Theorem \ref{teo6}. The proof of
Theorem \ref{teo6} is completed.
\end{proof}

\section{Product two-weighted weak-type estimates for $M_{\Omega ,\protect%
\alpha }^{\left( m\right) }$}

For the proof of Theorem \ref{teo53}, we need following lemma.

\begin{lemma}
\label{lemma52}\cite{Shi} Let $0<\alpha <mn$, $1\leq p_{i}\leq q_{i}<\infty $
for each $i=1,2,\ldots ,m$. Let $\frac{1}{p}=\frac{1}{q_{1}}+\frac{1}{q_{2}}%
+\cdots +\frac{1}{q_{m}}$ and $\left( u,\upsilon \right) $ be a pair of
weights. If $\left( u,\upsilon \right) \in \dbigcap
\limits_{i=1}^{m}A_{p_{i},q_{i}}^{\frac{\alpha }{m}}$, then for every $%
f_{i}\in L_{\upsilon }^{p_{i}}\left( {\mathbb{R}^{n}}\right) $, there is a
constant $C$, independent of $f_{i}$, such that%
\begin{equation*}
\left \Vert M_{\alpha }^{\left( m\right) }\left( \overrightarrow{f}\right)
\right \Vert _{L_{u}^{p,\infty }\left( {\mathbb{R}^{n}}\right) }\leq C\dprod
\limits_{i=1}^{m}\left \Vert f_{i}\right \Vert _{L_{\upsilon }^{p_{i}}\left( 
{\mathbb{R}^{n}}\right) }.
\end{equation*}
\end{lemma}

\textbf{The proof of Theorem \ref{teo53}.}

\begin{proof}
Since $s^{\prime }\leq p_{i}\leq q_{i}<\infty $ for each $i$ and $\left(
u,\upsilon \right) \in \dbigcap \limits_{i=1}^{m}A_{\frac{p_{i}}{s^{\prime }}%
,\frac{q_{i}}{s^{\prime }}}^{\frac{\alpha s^{\prime }}{m}}$, we can see by
Lemma \ref{lemma1*} and Lemma \ref{lemma52} that%
\begin{eqnarray*}
\left \Vert M_{\Omega ,\alpha }^{\left( m\right) }\left( \overrightarrow{f}%
\right) \right \Vert _{L_{u}^{p,\infty }\left( {\mathbb{R}^{n}}\right) }
&=&\sup_{\lambda >0}\lambda \left( u\left( \left \{ x:\left \vert M_{\Omega
,\alpha }^{\left( m\right) }\left( \overrightarrow{f}\right) \right \vert
>\lambda \right \} \right) \right) ^{\frac{1}{p}} \\
&\leq &C\sup_{\lambda >0}\lambda \left( u\left( \left \{ x:\left \vert \left[
M_{\alpha s^{\prime }}^{\left( m\right) }\left( \left \vert f_{1}\right
\vert ^{s^{\prime }},\left \vert f_{2}\right \vert ^{s^{\prime }},\ldots
,\left \vert f_{m}\right \vert ^{s^{\prime }}\right) \left( x\right) \right]
\right \vert >\lambda ^{s^{\prime }}\right \} \right) \right) ^{\frac{1}{p}}
\\
&\leq &C\sup_{\theta >0}\theta ^{\frac{1}{s^{\prime }}}\left( u\left( \left
\{ x:\left \vert \left[ M_{\alpha s^{\prime }}^{\left( m\right) }\left(
\left \vert f_{1}\right \vert ^{s^{\prime }},\left \vert f_{2}\right \vert
^{s^{\prime }},\ldots ,\left \vert f_{m}\right \vert ^{s^{\prime }}\right)
\left( x\right) \right] \right \vert >\theta \right \} \right) \right) ^{%
\frac{1}{p}} \\
&\leq &C\left \Vert M_{\alpha s^{\prime }}^{\left( m\right) }\left( \left
\vert f_{1}\right \vert ^{s^{\prime }},\left \vert f_{2}\right \vert
^{s^{\prime }},\ldots ,\left \vert f_{m}\right \vert ^{s^{\prime }}\right)
\right \Vert _{L_{u}^{\frac{p}{s^{\prime }},\infty }\left( {\mathbb{R}^{n}}%
\right) }^{\frac{1}{s^{\prime }}} \\
&\leq &C\dprod \limits_{i=1}^{m}\left \Vert \left \vert f_{i}\right \vert
^{s^{\prime }}\right \Vert _{L_{\upsilon }^{\frac{p_{i}}{s^{\prime }}}\left( 
{\mathbb{R}^{n}}\right) }^{\frac{1}{s^{\prime }}}\leq C\dprod
\limits_{i=1}^{m}\left \Vert f_{i}\right \Vert _{L_{\upsilon }^{p_{i}}\left( 
{\mathbb{R}^{n}}\right) }.
\end{eqnarray*}%
This is desired inequality (\ref{52}). We complete the proof of Theorem \ref%
{teo53}.
\end{proof}

\end{document}